\documentclass[sn-mathphys]{sn-jnl}

\jyear{2021}%

\theoremstyle{thmstyleone}%
\newtheorem{theorem}{Theorem}
\theoremstyle{thmstyletwo}%

\theoremstyle{thmstylethree}%
\newtheorem{definition}{Definition}%

\begin{document}

\title[Small Errors Imply Large Instabilities]{Small Errors Imply Large Instabilities}


\author*[1]{\fnm{Robert} \sur{Schaback}}\email{schaback@math.uni-goettingen.de}

\affil*[1]{\orgdiv{Institut f\"ur Numerische und Angewandte Mathematik},
  \orgname{Universit\"at G\"ottingen},
  \orgaddress{\street{Lotze\-stra\ss{}e 16--18},
    \city{G\"ottingen}, \postcode{37083}, 
    \country{Germany}}}


\abstract{Numerical Analysts and scientists working in applications
often observe that once they improve their techniques to get a better accuracy,
some instability creeps in through the back door. This paper
shows for a large class of numerical methods that
such a {\em Trade-off Principle} between error and stability
is unavoidable. It is an instance of a {\em no free lunch theorem}.
The setting is confined to recovery of functions
from data, but it includes solving differential equations by
writing such methods as 
a recovery of functions under constraints imposed by differential operators
and boundary values. It is shown in particular that
Kansa's Unsymmetric Collocation Method sacrifices accuracy for stability,
when compared to symmetric collocation.}




\maketitle

\section{Introduction}\label{SecIntro}
After quite some efforts to find kernels that
allow small recovery errors and well-conditioned kernel matrices
at the same time, the paper \cite{schaback:1995-1} proved that this does not
work. The result was called ``Uncertainty Relation''
or ``Trade-off Principle'' (see e.g. Holger Wendland's book
\cite{wendland:2005-1} of 2005) and received
quite some attention in the literature. It is a special case of the ``No free
lunch'' principle.
As correctly mentioned by Greg Fasshauer and Michael McCourt in
their 2015 book 
\cite{fasshauer-mccourt:2015-1},
it had quite some negative influence on the development of
the field, because it kept users from looking at
better bases than those spanned
by kernel translates. But it will turn out here that changes of bases will
not really help as long as the other ingredients are fixed.

Sparked by a question  of C.S. Chen of  the University of Southern
Mississippi in an e-mail dated Dec. 28th, 2021, this paper extends the result
of \cite{schaback:1995-1} to much more general situations. To avoid the
misconceptions implied by \cite{schaback:1995-1}, the effect of basis changes
will be discussed at various places. But most of the results here
are independent of choices of bases.

Since the scope of the paper will be quite wide,
a good deal of abstraction will be necessary later, and therefore
a classical case should be served as starters. 
Consider interpolation of functions $f\in C^{n+1}[-1,+1]$
on a set $X_n$ of points $-1\leq x_0<x_1,\ldots <x_n\leq 1$ by polynomials
$I_n(f)$ of degree at most
$n$.
The well-known error bound is
$$
\lvert f(x)-I_n(f)(x)\rvert\leq \dfrac{\|f^{(n+1)}\|_\infty}{(n+1)!}
\prod_{j=0}^n\lvert x-x_j\rvert
$$
based on Newton's formula. 
We can recast this
as an error bound
$$
\lvert f(x)-I_n(f)(x)\rvert\leq P_{X_n}(x)  \|f^{(n+1)}\|_\infty
$$
in terms of a {\em Power Function} 
$$
\begin{array}{rcl}
P_{X_n}(x)&:=&\displaystyle{ \sup_{\|f^{(n+1)}\|_\infty\leq 1}
\lvert f(x)-I_n(f)(x)\rvert}\\
&=&
\displaystyle{\frac{1}{(n+1)!}
 \prod_{j=0}^n\lvert x-x_j\rvert}.
\end{array} 
$$

If we add a point $x$ to the set $X_n$, the {\em Lagrangian} 
of degree $n+1$, vanishing on $X_n$
and being one on $x$ is
$$
u_{x,X_n}(z)=\prod_{j=0}^n\frac{z-x_j}{x-x_j}
$$
with seminorm
$$
\|u^{(n+1)}_{x,X_n}\|_\infty =(n+1)!\prod_{j=0}^n\lvert x-x_j\rvert ^{-1},
$$
leading to
\begin{equation}\label{eqPLmfUpoly} 
1=P_{X_n}(x)\cdot \|u^{(n+1)}_{x,X_n}\|_\infty.
\end{equation}
This is a {\em Trade-off Principle}:
\begin{quote}
 Small Power Functions lead to large norms of  Lagrangians.
\end{quote}
Since Lagrangians are the images of unit data in the function space,
large norms of  Lagrangians lead to large norms of interpolation operators
as maps from data to functions. Then small data variations lead to
large variations in the resulting functions, and one may call this
a grade of {\em evaluation instability}. Thus the Trade-off Principle
implies
\begin{quote}
 Small errors lead to large evaluation instabilities.
\end{quote}
The paper gives a rigid underpinning to this somewhat sloppy statement.
Recall that {\em regularization} of operator equations
works exactly in the same way: part of the recovery error
is sacrificed for better stability. 

However, the numerical computation of the Lagrangians induces
additional instabilities that are ignored here. To cope with these,
{\em barycentric} formulas were introduced for the polynomial case, see
J.P. Berrut and L.N. Trefethen \cite{berrut-trefethen:2004-1}.
For kernel-based recoveries, various methods were invented to cope with
instabilities, see the references given in Section \ref{SecOut}.

Section \ref{SecDaF} sets the stage for general
recovery methods including solving differential equations.
{\em Recovery processes} reconstruct functions from {\em data}
given as prescribed values of {\em linear functionals}, and  
the evaluation of the result will again be an application of a functional.
{\em Data} can include values of arbitrary linear operators
acting on functions, thus
rewriting methods for PDE solving as function recoveries.

Section \ref{SecRoF} introduces the form
of the recoveries considered. Nearest-neighbor methods, optimal recoveries
in Hilbert spaces, and regression in Machine Learning are
special cases described in Section \ref{SecSRS}.

The basic technique used for trade-off principles is outlined in Section
\ref{SecDTOP}, still in rather abstract
form. 
Then Section \ref{SecEaS} contains the central results, namely
trade-off principles that bound the  product of norms of errors and
norms of certain worst-case functions from below. It is shown how
the latter govern instability of the evaluation of the recovery.
The lower bounds turn into equalities in case of optimal recoveries
in Section \ref{SecORoF}.

Examples are given in Section \ref{SecExa}, including splines and 
recoveries via expansions like Fourier or Taylor series.
The connection to the older trade-off principle from \cite{schaback:1995-1}
is provided in Section \ref{SecKBRP}, followed by extensions
to unsymmetric methods like Kansa's collocation technique.
The trade-off principle holds for these as well,
but they sacrifice accuracy for evaluation stability.
Finally, the implications
for greedy adaptive methods are  sketched. 
\section{Data as Functionals}\label{SecDaF}
A fairly general and useful viewpoint on Numerical Analysis or
Computational Mathematics when working on functions is to
see {\em data} of a function as {\em values of linear functionals}.
In particular, differential equations, ordinary or partial,
just impose infinitely many restrictions on a function $u$ from some function
space $U$ by applying linear functionals. This can be conveniently written as
\begin{equation}\label{eqluf}
\lambda(u)=f_\lambda\in \mathbb{R} 
\end{equation}
for all functionals $\lambda$ from a subset $\Lambda$ of the dual $U^*$ of $U$,
the space of continuous linear functionals on $U$. The problem is to recover $u$
from the given data
$f_\lambda$. The specifics of certain
differential equation problems involving differential or boundary evaluation
operators disappear. And if users have only limited information in the sense
of just finitely many data $f_{\lambda_j}\in \mathbb{R}$ for a finite subset
$\Lambda_M=\{\lambda_1,\ldots,\lambda_M\}\subset \Lambda$, one has to use
computational techniques that get along with the available data.
This viewpoint is behind the scenes for this paper. Readers should always be
aware that differential operators may lurk behind the functionals
appearing here.

For illustration, consider a standard Poisson problem
\begin{equation}\label{eqPoisson}
\begin{array}{rcll}
\Delta u &=& f & \hbox{ in } \Omega\\ 
       u &=& g & \hbox{ in } \partial\Omega\\ 
\end{array}
\end{equation}
on a bounded domain $\Omega\subset \mathbb{R}^2$
for simplicity. The data functionals
come in two variations:
\begin{equation}\label{eqPoissonDisc}
\begin{array}{rclll}
\Delta_i(f)&:=& \Delta f(x_i),&x_i\in \overline\Omega,&1\leq i\leq M_\Delta\\
\beta_j(f)&=& f(y_j),&y_j\in \partial\Omega,&1\leq j\leq M_\beta\\
\end{array}
\end{equation}
caring for the PDE in the domain and for the boundary values.
They are finite selections from the obvious infinite
sets of functionals that define the true solution analytically.
If the analytic problem is well-posed and if the function recovery from
the above data is carried out with enough oversampling,
this technique produces accurate and convergent
approximations to the true solution
of the PDE problem \cite{schaback:2016-4}. This reference
also fits algorithms that solve
problems in weak form into this framework,
including the Meshless Local Petrov Galerkin approach
by S.N. Atluri and T.-L. Zhu 
\cite{atluri-zhu:1998-2,atluri-zhu:1998-1} and Generalized Finite Element
Methods, see the
survey by I. Babu\v{s}ka et.al. \cite{babuska-et-al:2003-1}.
The general practical observation is that going for more accuracy
causes more instabilities, in a way that will be clarified here. 

Also, {\em evaluation} of functions is the
application of a functional $\mu\in U^*$ to some function $f$.
In particular, evaluation of a
multivariate derivative $D^\alpha$ at a point $x$ is the
application of the functional 
$\delta_x(D^\alpha f)=f^{(\alpha)}(x)$ in case
that $\delta_x D^\alpha$ is continuous on $U$.
If point evaluation is
not defined, as in $L_2$ spaces, but if local integration is feasible,
one can evaluate 
local integral means, as substitutes for
point evaluation. This is the standard way to handle
problems in weak form in the references cited above.

Summarizing, everything boils down
to a matter of functionals. What can we say about $\mu(f)$ if we know all
$\lambda(f)$ for all
$\lambda\in \Lambda$? In particular, what can we say about $f(x)$ when we know
plenty of data $f(x_j)$? Note that this problem is {\em  regression}
in a probabilistic context, and it arises in Machine Learning on a
large scale, with Big Data given in high-dimensional spaces.
\section{Recovery of Functions}\label{SecRoF}
We now postulate that we
can write the recovery of functions $f$ from their data
$\Lambda(f)=(\lambda_1(f),\ldots,\lambda_M(f))^T\in \mathbb{R}^M$
as a linear {\em recovery map} 
\begin{equation}\label{eqalm}
f \mapsto R_{a_\Lambda}(f):=a^T_\Lambda\Lambda(f)=
\displaystyle{\sum_{\lambda_j\in \Lambda}a_{\lambda_j} \lambda_j(f)   }
\hbox{ for all } f\in U
\end{equation}
such that the span of the elements $a_{\lambda_j}$
of the vector $a_\Lambda\in U^M$ defines a {\em trial subspace}
of functions in $U$. To avoid certain complications,
the map $f\to \Lambda(f)\in \mathbb{R}^M$ is assumed to be surjective.

We call
the recovery process {\em interpolatory} if
the recovery preserves the data, i.e.
$$
\Lambda(R_{a_\Lambda}(f))=\Lambda(a_\Lambda^T\Lambda(f))=\Lambda(f) \hbox{ for all } f\in U,
$$
and then $\lambda_k(a_{\lambda_j})=\delta_{jk},\,1\leq j,k\leq M$ holds
and the $a_{\lambda_j}$ form a {\em Lagrange basis} with Kronecker data.
We shall use the notation $u_{\mu,\Lambda}$ for a {\em Lagrangian}
that satisfies $\Lambda(u_{\mu,\Lambda})=0,\;\mu(u_{\mu,\Lambda})=1$,
and then $a_{\lambda_j}=u_{\lambda_j,\Lambda\setminus\{\lambda_j\}}$ holds in this
notation.

Lagrangians will not exist for general recoveries. But the $a_{\lambda_j}$
may be called
 {\em pseudo-Lagrangians}
 because they produce the recovery like Lagrangians, but
 without exact reproduction of the data.

Evaluation of the recovery via a
functional $\mu$ now is
$$
\mu(a^T_\Lambda\Lambda(f))=a^T_\Lambda(\mu)\Lambda(f)=
\displaystyle{\sum_{\lambda_j\in \Lambda}\mu(a_{\lambda_j}) \lambda_j(f)   }
$$
by defining a vector $a^T_\Lambda(\mu):=\mu(a^T_\Lambda)$
that is a bilinear form in $\Lambda$ and $\mu$.

One may restrict these maps to sums over neighbours $\lambda$ of $\mu$,
to get more {\em locality}, and this is what generalized
{\em Moving Least Squares}
\cite{farwig:1986-1,levin:1998-1,wendland:2000-1, armentano:2001-1}
or Finite Elements do.
But if theoretically
done for all $\mu\in U^*$, this still fits into the above framework.
As a prominent example, {\em Barycentric formulas} by J.P. Berrut and L.N. Trefethen
\cite{berrut-trefethen:2004-1} change the way
the above formula is calculated, with a significant gain in
numerical stability. 

For recovery of a single value  $\mu(f)$ from single values $\lambda_j(f)$
one can construct a vector of single values $a_{\lambda_j}(\mu)$ such that
\begin{equation}\label{eqmuff}
\mu(f)\approx
\displaystyle{\sum_{\lambda_j\in \Lambda}a_{\lambda_j}(\mu) \lambda_j(f)   }
\end{equation}
without necessarily calculating the $a_{\lambda_j}$ as functions and taking
values $\mu(a_{\lambda_j})=a_{\lambda_j}(\mu)$ afterwards.
In meshless methods (see the early survey by
T. Belytschko et.al. \cite{belytschko-et-al:1996-1}),
the functions $a_{\lambda_j}$ are called {\em shape functions}. In the standard
approach, they are calculated in many points, and if derivatives are
needed for dealing with PDEs, these are taken afterwards or obtained by taking
derivatives of the local construction process.
In contrast to this, {\em Direct Moving Least Squares}
by D. Mirzaei et.al. \cite{mirzaei-et-al:2012-1},
\cite{mirzaei-schaback:2013-1} use (\ref{eqmuff}) for derivative
functionals $\mu$ without the detour via shape functions. 

This presentation looks unduly abstract, but it isn't. It considers
recovery without any fixed assumptions about how functions are represented,
how norms of errors and functions are defined, and how bases are chosen.
Therefore it allows to compare
actual numerical strategies on a higher level. It goes back to the input data
and considers the output data, as functionals,
the actual determination of the recovery map
being in a black box.
The final goal in this paper is to see whether
going for a small error implies
some sort of instability whatsoever, and this may be independent
of what happens in the black box. This is why we consider
errors and stability in section \ref{SecEaS}
after we present some examples.
\section{Special Recovery Strategies}\label{SecSRS}
When avoiding full functions, the recovery of a value $\mu(f)$
from given values $\Lambda(f)$ 
is trivial  if $\mu =\lambda$ for
some $\lambda\in\Lambda$. In more generality,
one would pick the functional $\lambda\in \Lambda$ that is
``closest'' to $\mu$, and then take $\lambda(f)$ as an approximation of
$\mu(f)$. This is the {\em nearest neighbour} strategy, but it needs 
distances between functionals, and requires to find the
closest neighbour. Cases involving point geometry
like nearest neighbours or triangulations will be covered
by the theory developed here, but we do not include examples.

If a norm on $U^*$
is available, one can
consider the approximation problem to minimize
$$
\displaystyle{\left\|\mu
  -\displaystyle{\sum_{\lambda\in \Lambda}a_\lambda \lambda   }
  \right\|_{U^*}   } 
$$
over all coefficients $a_\lambda$, denote a solution by $a^*_\lambda(\mu)$
and to approximate $\mu$  by
$$
\mu_\Lambda^*:=\displaystyle{\sum_{\lambda\in \Lambda}a^*_\lambda(\mu) \lambda   }.
$$
This avoids functions as well, but it requires
norms in the dual space that users can work with.

Special cases are
Reproducing Kernel Hilbert spaces. They have a kernel
$K\;:\;\Omega\times\Omega$ on an abstract set $\Omega$
and define an inner product
$$
(\lambda,\mu)_{U^*}:=\lambda^x \mu^y K(x,y)
$$
where the application on $x$ arises as a superscript. Furthermore, each
functional $\lambda$ defines a function
$$
f_\lambda(x):=\lambda^yK(y,x) \hbox{ for all } x\in\Omega
$$
and these functions have the inner product
$$
(f_\lambda, f_\mu)_U=(\lambda,\mu)_{U^*}=\lambda^x \mu^y K(x,y)
$$
making $f_\lambda$ a Riesz representer of $\lambda$.
It is then easy to prove that an optimal recovery consists of the vector
$a_\lambda^*(\mu)$ that solves the system
$$
\left(
\begin{array}{cccc}
  (\lambda_1,\lambda_1)_{U^*} &(\lambda_1,\lambda_2)_{U^*} &\ldots
  & (\lambda_1,\lambda_M)_{U^*} \\
  (\lambda_2,\lambda_1)_{U^*} &(\lambda_2,\lambda_2)_{U^*} &\ldots
  & (\lambda_2,\lambda_M)_{U^*} \\
  \vdots & \vdots &\ddots& \vdots\\
  (\lambda_M,\lambda_1)_{U^*} &(\lambda_M,\lambda_2)_{U^*} &\ldots
  & (\lambda_M,\lambda_M)_{U^*} \\  
\end{array}
\right)
\left(
\begin{array}{c}
a^*_{\lambda_1}(\mu) \\
a^*_{\lambda_2}(\mu) \\
\vdots\\
a^*_{\lambda_M}(\mu) \\
\end{array}
\right)
=
\left(
\begin{array}{c}
({\lambda_1},\mu)_{U^*} \\
({\lambda_2},\mu)_{U^*} \\
\vdots\\
({\lambda_M},\mu)_{U^*} \\
\end{array}
\right)
$$
with a {\em kernel matrix} that usually is positive definite.

This looks theoretical again, but it applies to Sobolev spaces,
having Whittle-Mat\'ern kernels, and therefore it is useful
for solving PDE problems by recovery of functions from PDE data.
This recovery strategy has various optimality properties
\cite{schaback:2015-3} that we skip over
here. See details on kernel-based methods in books by M.D. Buhmann
\cite{buhmann:2003-1}, H. Wendland \cite{wendland:2005-1},
and G. Fasshauer/M. McCourt \cite{fasshauer-mccourt:2015-1}.

It also applies to Machine Learning
\cite{schaback-wendland:2006-1}. On a general set $\Omega$ one has
{\em feature maps} $\varphi_n$ that map the abstract objects $x$
to a real value like cost or weight or area. The kernel then is
$$
K(x,y)=\displaystyle{\sum_n \rho_n \varphi_n(x)\varphi_n(y)}
$$
with positive weights $\rho_n$ and the inner product
$$
(\lambda,\mu):=\displaystyle{\sum_n
  \dfrac{\lambda^x\varphi_n(x)\mu^y\varphi_n(y)}{\rho_n}   } 
$$
lets the above machinery work for regression, but details
are omitted. Combining the cases above, Machine Learning
can ``learn'' the solution of a PDE using this framework.
\section{Dual Trade-off Principles}\label{SecDTOP}
Throughout we shall assume that norms in $U$ and $U^*$
are defined and connected via the suprema
\begin{equation}\label{eq2xsup}
\displaystyle{  \|f\|_U= \sup_{0\neq \mu\in U^*}\frac{\mu(f)}{\|\mu\|_{U^*}}
  ,\;\;
 \|\mu\|_{U^*}= \sup_{0\neq f\in U}\frac{\mu(f)}{\|f\|_{U}}.
}
\end{equation}
Take a functional $\mu\in U^*$ and imagine that it is evaluating 
an error of a recovery process. Then
$$
1\leq \|\mu\|_{U^*}\cdot \|f_\mu\|_U
$$
holds for all functions $f_\mu\in U$ with $\mu(f_\mu)=1$. If the error
$\|\mu\|_{U^*}$ of the recovery process is small, the norms of
the functions $f_\mu$ must be large. We shall later
interpret this as an instability of the evaluation of
the recovery operator.

Of course, there also is a dual version
$$
1\leq \|f\|_{U}\cdot \|\mu_f\|_{U^*}
$$
for all functions $f \in U$ and all functionals $\mu_f\in U^*$
with $\mu_f(f)=1$. In Hilbert spaces, one can minimize the second factors
under the given constraint, and the minimum is realized by Riesz
representers. 

Note that the above inequalities turn into equalities if the suprema in
(\ref{eq2xsup}) are attained for the functions or functionals
in the second factors. 
Details and applications will follow below.

It is essential that the two norms in the above inequalities
are dual to each other. If $\|.\|_U$ allows and penalizes high derivatives,
the functionals in $U^*$ will allow high derivatives as well.
Users might want the error factor and the stability factor to use non-dual
norms, but this is a quite different story. 

These trade-off principles differ from certain
standard algebraic ones like $1\leq \|A\|\|A^{-1}\|$ for
square nonsingular matrices $A$, or
\begin{equation}\label{eqxy2}
\lvert(x,y)_2\rvert\leq \|x\|_2\|y\|_2
\end{equation}
for vectors. If generalized to
variances and a covariance or commutator, the latter case
is behind the Heisenberg Uncertainty Principle after a few steps of
generalization. In contrast to this, 
dual norms come into play here, and the trade-off principles will 
hold for all choices of norms.
\section{Error and Stability}\label{SecEaS}
We assume to have a finite set $\Lambda\subset U^*$ of functionals
to recover functions $f\in U$ from their data $\Lambda(f)$
via a recovery (\ref{eqalm}), and
we evaluate the result by applying a functional $\mu\notin \Lambda$ to $f$. 
\begin{definition}\label{DefPFgen}
The norm
$$
P_{a_\Lambda}(\mu):=\left\|\mu-\mu(R_{a_\Lambda})\Lambda \right\|_{U^*}
$$
is called the {\em Generalized Power Function}.
\end{definition}
It leads to an
error bound
\begin{equation}\label{eqmTmf}
\begin{array}{rcl}
\lvert \mu(f)-\mu(R_{a_\Lambda})(f)\rvert
&\leq&
P_{a_\Lambda}(\mu)\|f\|_U \hbox{ for all } f\in U,\;\mu\in U^*
\end{array} 
\end{equation}
and this is why we use it to deal with the recovery error.
\begin{definition}\label{DefBFgen}
A {\em bump function} $f_{\mu,\Lambda}\in U$ satisfies 
$\mu(f_{\mu,\Lambda})=1$ and $\Lambda(f_{\mu,\Lambda})=0$.
\end{definition}
\begin{theorem}\label{TheUP}
  For any any functional $\mu\in U^*$
  that has a bump function $f_{\mu,\Lambda}$,  
  the trade-off principle
\begin{equation}\label{eqPLmfU}
1\leq P_{a_\Lambda}(\mu)\|f_{\mu,\Lambda}\|_U
\end{equation}
holds, and
\begin{equation}\label{eqiPif}
\displaystyle{\frac{1}{\inf_{a_\lambda}P_{a_\Lambda}(\mu)}
\leq \inf_{\mu(f)=1,\Lambda(f)=0}\|f\|_U} 
\end{equation}
relates the best possible recovery to the best possible
bump function.
\end{theorem}
\begin{proof}
  Just insert a bump function into (\ref{eqmTmf}).
\end{proof}
{\bf Remark}: Power Functions, bump functions, and Lagrangians
are independent of bases. There is no escape from the Trade-off Principle by any
change of basis, as long as the recovery map or the space $U$ are kept fixed.

{\bf Remark}: Furthermore, recoveries, bump functions, and Lagrangians can be 
defined without using norms or spaces. These come up when
going over to a Power Function and a norm of a bump function.
Therefore Theorem \ref{TheUP} does not only cover all
possible recoveries, but also all ways to handle
errors and evaluation instability for these
by defining norms afterwards. Furthermore,
(\ref{eqPLmfU}) is {\em local} in the sense that it
holds for each specific $\mu$. The right-hand side will vary
considerably with $\mu$, up to the limit $1\leq 0\cdot\infty$ in the excluded
case $\mu\in \Lambda$. 

While (\ref{eqPLmfU}) is an {\em add-one-in} version,
a special {\em leave-one-out} version is
\begin{equation}\label{eqPLlulf}
1\leq P_{a_{\Lambda\setminus\{\lambda\}}}(\lambda)\|f_{\lambda,\Lambda\setminus\{\lambda\}}\|_U
\end{equation}
if a bump function $f_{\lambda,\Lambda\setminus\{\lambda\}}$ is available.
And if
the recovery is interpolatory, using a Lagrangian
$u_{\lambda,\Lambda\setminus\{\lambda\}}$  we get
\begin{equation}\label{eqPLlulL}
1\leq P_{\Lambda\setminus\{\lambda\}}(\lambda)\|u_{\lambda,\Lambda\setminus\{\lambda\}}\|_U.
\end{equation}
When using the leave-one-out version, it is a pitfall to assume
that the recovery $a_{\Lambda\setminus \{\lambda\}}$ arises from
deleting the component $a_\lambda$ from $a_\Lambda$. The other components
will still depend on all functionals in $\Lambda$.

We now have to show that the second factor governs the stability of evaluation
of the recovery. 
The norm of the interpolation as a map from data to functions in $U$
is blown up
for large $U$-norms of Lagrangians due to
$$
\begin{array}{rcl}
  \max_{\lambda\in \Lambda}\|u_{\lambda,\Lambda\setminus\{\lambda\}}\|_U
  &\leq&
  \displaystyle{ \sup_{0\neq f\in U}\dfrac{ \|R_{a_\Lambda}(f)\|_{U}}{\|\Lambda(f)\|}}\\
&=& \|R_{a_\Lambda}\|_{\Lambda(U),U}\\
&\leq &
  \displaystyle{\sup_{0\neq f\in U}\dfrac{ \|\sum_{\lambda\in \Lambda}\lambda(f)
    u_{\lambda,\Lambda\setminus\{\lambda\}}  \|_{U}}{\|\Lambda(f)\|}}\\
&\leq &
  \left(\max_{\lambda\in \Lambda}\|u_{\lambda,\Lambda\setminus\{\lambda\}}\|_U\right)
  \left(\sum_{\lambda\in\Lambda}\|\lambda\|_{U^*}\right).
\end{array} 
$$
Summing up (\ref{eqPLlulL}) in the interpolatory case,
we get a trade-off principle
$$
\begin{array}{rcl}
  \lvert \Lambda\rvert 
  &\leq &
  \| P_{\Lambda\setminus\{\lambda\}}(\lambda)  \|_{p,\mathbb{R}^{\lvert \Lambda \rvert}}
  \| \|u_{\lambda,\Lambda\setminus\{\lambda\}}\|_U  \|_{q,\mathbb{R}^{\lvert \Lambda \rvert}}\\
\end{array}
$$
that lets the final factor grow when the error is small. Here, the
norms in $\mathbb{R}^{\lvert \Lambda \rvert}$ are running over the $\lambda\in \Lambda$,
and we allow $1/p+1/q=1$.

Assume that the data $\lambda(f)$ for a single $\lambda$
and a fixed function $f$ carries an absolute error $\epsilon$. Then the results
of (\ref{eqalm}) will differ by $\epsilon a_\lambda(\mu)=\epsilon \mu(a_\lambda)$
showing that evaluation and its expectable roundoff blows up
with increasing pseudo-Lagrangians $a_\lambda$. 
In the interpolatory case,
$a_\lambda(\mu)=\mu(u_{\lambda,\Lambda\setminus\{\lambda\}})$ implies that 
the input errors propagate by the Lagrangians into the result,
and (\ref{eqPLlulL}) is a lower bound of the product between error and evaluation
stability. 
This, again, is why Lagrangians are closely connected to stability
of the evaluation of an interpolant.
Even if Lagrangians are never calculated, they are behind the scene
in any interpolatory recovery, if the final evaluation is a
weighted sum (\ref{eqalm}) over the $\lambda(f)$ when done exactly.
This holds because the
$a_\lambda(\mu)$ will always be values of Lagrangians,
even if the latter are
avoided by tricky numerical detours.
The stability of methods
for calculating Lagrangians is ignored here.

So far it is not clear what bump functions have to do
with the stability of evaluation. In general,
$$
\|R_{a_\Lambda}(f_{\lambda,\Lambda\setminus\{\lambda\}})\|_U=\|a_\lambda\|_U
\leq \|R_{a_\Lambda}\|\|f_{\lambda,\Lambda\setminus\{\lambda\}}\|_U
$$
shows that control over norms of bump functions implies control
over the $\|a_\lambda\|_U$, and we saw above that these blow up
absolute errors in the input data.
If the recovery operator $R_{a_\Lambda}$ 
is used without changes to evaluate the recovery result,
the evaluation is bounded above by
\begin{equation}\label{eqTLmfS}
\begin{array}{rcl}
  \lvert \mu(R_{a_\Lambda})(f)\rvert
  &=&
\displaystyle{\left\lvert \sum_{\lambda\in
    \Lambda}\mu(a_{\lambda})\lambda(f) \right\rvert  } \\
&\leq &
\|\lambda(f)\|_{p,\mathbb{R}^{\lvert \Lambda \rvert}} \|\mu(a_\lambda)    \|_{q,\mathbb{R}^{\lvert \Lambda \rvert}}\\
&\leq &
\|\lambda(f)\|_{p,\mathbb{R}^{\lvert \Lambda \rvert}} \;\|\mu\|_{U^*}
\;\|\|a_\lambda\|_U     \|_{q,\mathbb{R}^{\lvert \Lambda \rvert}}\\
\end{array}
\end{equation}
with $1/p+1/q=1$ and where the norms on $\mathbb{R}^{\lvert \Lambda \rvert}$ run over the $\lambda$
values.
The bound factors into the linear influence of $f$ and $\mu$
and keeps the final factor as something like a Lebesgue constant.

\noindent This implies, in a somewhat sloppy formulation,
the trade-off principle
\begin{quote}
Small errors imply large evaluation instabilities
\end{quote}
that tacitly assumes that errors are measured via norms in $U^*$
while evaluation instabilities are measured via
norms in $U$, the two being dual
in the sense of section \ref{SecDTOP}.

Dealing with a non-dual situation, in particular with a weaker notion
of evaluation instability, requires much more machinery. 
A typical case is $L_\infty$ evaluation stability governed by
$L_\infty$ norms of Lagrangians 
for point evaluation data $\lambda_j(f)=f(x_j)$. 
This is the standard path along Lebesgue functions and Lebesgue constants.
For univariate polynomial interpolation,
this approach reveals that equidistant points
have an exponential $L_\infty$ instability, while Chebyshev-distributed points
only have a logarithmic $L_\infty$ instability. For
kernel-based interpolation of function values,
\cite{demarchi-schaback:2010-1} proved uniform boundedness of
Lagrangians if point sets are asymptotically uniformly distributed
and if kernels have finite smoothness. Since errors can differ
between such kernels while Lagrangians are always uniformly bounded,
there is no strict Trade-off Principle under these circumstances. 

{\bf Remark}: Certain numerical techniques put a map $C$ and
its inverse into (\ref{eqTLmfS}) like
$$
  R_{a_\Lambda}(\mu)(f)
  =
\displaystyle{ \left(\sum_{\lambda\in
    \Lambda}\mu(a_{\lambda})\sum_{\tau}C^{-1}_{\lambda,\tau}
  \right)\left(
  \sum_{\lambda}C_{\tau,\lambda}\lambda(f) \right) }.
$$
The classical example is standard univariate polynomial interpolation using a transition to
divided differences on the functional side and to the Newton basis
on the function side. The stability properties of the recovery
map as a whole
are not changed by that. Possible instabilities are just distributed
over both factors. These effects are ignored here.

\section{Optimal Recovery of Functions}\label{SecORoF}
We shall see interpolatory cases
where the Lagrangian satisfies (\ref{eqPLmfU})
and (\ref{eqPLlulL})
with equality.
Then, by (\ref{eqiPif}),
the norms of the Lagrangians are minimal under the
norms of all bump functions, and the recovery process
has minimal error.
Under certain conditions satisfied for splines and
kernel-based interpolation, this holds systematically:
\begin{theorem}\label{TheExTOP}
  Assume that an interpolatory recovery process $R_{a_\Lambda}$ satisfies
  minimum-norm and bounded-error properties, i.e.
  $$
  \|R_{a_\Lambda}(f)\|_U=\inf_{g\in U, \Lambda(f)=\Lambda(g)}\|g\|_U \leq
 \|f\|_U  \hbox{ and }  \|f-R_{a_\Lambda}(f)\|_U\leq \|f\|_U
  \hbox{ for all } f\in U.
  $$
  Then (\ref{eqPLmfU}) and (\ref{eqiPif}) hold with equality.
  The Lagrangian is
  the minimum-norm bump function, and the recovery process has a minimal Power
  Function. 
\end{theorem} 
\begin{proof}
We reformulate the Power Function via
$$
\begin{array}{rcl}
  P_{a_\Lambda}(\mu)
  &=&
  \displaystyle{ \sup_{\|f\|_U\leq 1}
\lvert \mu(f)-\mu(R_{a_\Lambda}(f))\rvert}\\
  &=&
  \displaystyle{ \sup_{\|f-R_{a_\Lambda}(f)\|_U\leq 1}
\lvert \mu(f)-\mu(R_{a_\Lambda}(f))\rvert}\\
\end{array} 
$$
where  $\geq$ follows from replacing $f$ by $f-R_{a_\Lambda}(f)$, and $\leq$
follows from the minimum error property, because the set of
$\|f-R_{a_\Lambda}(f)\|_U\leq 1$ contains the set of $\|f\|_U\leq 1$.
Then we go on by
$$
\begin{array}{rcl}
  P_{a_\Lambda}(\mu)
&=&
  \displaystyle{ \sup_{\|g\|_U\leq 1, \Lambda(g)=0}
\lvert\mu(g)\rvert}\\
\end{array}
$$
and note that $g$ van be written as $\mu(g)f_{\mu,\Lambda}$
for an arbitrary bump function $f_{\mu,\Lambda}$. Now
$$
\begin{array}{rcl}
  P_{a_\Lambda}(\mu)
&=&
  \displaystyle{ \sup_{\|f_{\mu,\Lambda}\|_U\leq 1/\lvert\mu(g)\rvert}\lvert\mu(g)\rvert}\\
  &=& \displaystyle{  \frac{1}{\min{\|f_{\mu,\Lambda}\|_U}}}\\
  &=& \displaystyle{ \frac{1}{\|u_{\mu,\Lambda}\|_U}}
\end{array}
$$
where the final line follows from the minimum norm property.
\end{proof}
\section{Examples of Interpolatory Recoveries}\label{SecExa}
This section illustrates how the Trade-off Principle works
under various circumstances. For the univariate cases in this section,
we treat interpolation in $[-1,+1]$ of values $f(x_j)$ of functions
$f$ on points $x_0,\ldots,x_n$
and evaluation at some point $x$. The functionals are $\delta$ functionals,
and everything can be expressed via the points. Other cases stay with
the original formulation via general functionals $\lambda_j$. 
\subsection{Connect-the-dots}\label{SecPWL}
The simplest univariate case is
{\em connect-the-dots} piecewise linear interpolation,
but it has no choice of a space $U$ yet.
The simplest is $U=C[-1,+1]$ under the sup norm,
the Lagrangians being hat functions, with constant extensions to
the boundary, if boundary points are not given.
Then there is no proper error bound, and (\ref{eqPLlulL})
consists of all ones.

If we keep the interpolation method as is,
we can go over to zero boundary values and 
$U=C_0^1[-1,+1]$ under the sup norm of the first derivative.
An add-one-in  Lagrangian $u_{x,x_k,x_{k+1}}$ based on three adjacent
increasingly ordered points $x_{k}<x<x_{k+1}$
will have the norm
$(\min(x_{k+1}-x,x-x_{k}))^{-1}$.
The add-one-in Power Function on some $x\in [x_k,x_{k+1}]$ is
$$
\begin{array}{rcl}
  P_{a_\Lambda}(x)
  &=&
  \sup_{\|f'\|_\infty\leq 1}\lvert f(x)-R_{a_\Lambda}(f)(x)\rvert\\
  &=&
  \sup_{\|f'\|_\infty\leq 1}
  \left\lvert \frac{x_{k+1}-x}{x_{k+1}-x_k}(f(x)-f(x_k))
  +\frac{x-x_k}{x_{k+1}-x_k}(f(x)-f(x_{k+1}))     \right\rvert\\
  &=&
  2\frac{(x_{k+1}-x)(x-x_k)}{x_{k+1}-x_k}
\end{array}
$$
and for the Trade-off Principle we get 
$$
\begin{array}{rcl}
 \frac{P_{a_\Lambda}(x)}{\|u_{x,x_k,x_{k+1}}\|_{C_0^1[-1,+1]}}=
 2\frac{(x_{k+1}-x)(x-x_k)}{(x_{k+1}-x_k)\min(x_{k+1}-x,x-x_{k})}
 =2\frac{\max(x_{k+1}-x,x-x_{k})}{x_{k+1}-x_k}
\end{array}
$$
being between 1 and 2.

If $U$ takes the $L_2$ norm of first derivatives,
we are in a standard spline
situation \cite{ahlberg-et-al:1967-1}
and Theorem \ref{TheExTOP} applies. This illustrates that
the Trade-off Principle works locally and for all possible norms
when the recovery problem is fixed to be connect-the-dots. 
\subsection{Taylor Data}\label{SecTD}
Here is a rather academic but mathematically interesting case.
Take a space $U$ of univariate
real-valued functions on $(-1,+1)$ that have complex extensions being
analytic in the unit disc, and consider {\em Taylor data}
functionals
$\lambda_j(f):=f^{(j)}(0)/j!$ for $j\geq 0$. Then write the functions
by their Taylor series
$$
f(z)=\displaystyle{\sum_{n\geq 0}\lambda_j(f)z^j   } 
$$
and define a norm in $U$ by
$$
\|f\|_U^2=\displaystyle{\sum_{n\geq 0}\frac{\lambda_j(f)^2(j!)^2}{\rho_j}   } 
$$
where the positive weights $\rho_j$ satisfy the constraint
\begin{equation}\label{eqTayCon}
\displaystyle{\sum_{n\geq 0}\frac{\rho_j}{(j!)^2}<\infty}. 
\end{equation}
This generates a Hilbert space of functions whose
reproduction formula is the
Taylor series, see \cite{zwicknagl-schaback:2013-1} for plenty
of examples, including Hardy and Bergman spaces. The interpolation here
is just a partial sum of the Taylor series, while it works by
kernel translates in \cite{zwicknagl-schaback:2013-1}.

Now the monomials $z^j$ are the Lagrangians for the $\lambda_j$,
with norms $\|z^j\|^2_U={(j!)^2}/{\rho_j}$. And Theorem
\ref{TheExTOP} holds because we just chop the Taylor series.
Consequently, inequalities (\ref{eqPLmfU}) and (\ref{eqPLlulL})
are satisfied by equality, and we also know that the Power Function
in the leave-last-out form is
$$
 P_{\Lambda_{k-1}}(\lambda_k)=\sqrt{\rho_k}/k!.
 $$
 We treat the add-one-in case in the next section.
 
  Note that all cases will behave like that if they take
  expansions into series of Lagrangians as their underlying space $U$,
  with weights for the expansion coefficients.
\subsection{Orthogonal Series}\label{SecOS}
Now assume that a space $U$ carries an inner product and allows an
orthonormal basis $u_0,u_1,\ldots$, while the data functionals
are $\lambda_j(u)=(u,u_j)_U$. Again, the Lagrangians are
the expansion basis, and we have Fourier series as a prominent example.
If only such functionals are considered,
this is a trivial case, because all Lagrangians and Power Functions
have norm one.

Now let $\mu$ be a different functional, and
we construct the norm-minimal
bump function of the form
$$
f_{\mu,\Lambda_n}=\displaystyle{\sum_{j>n}a_ju_j   }, \;a_j=(f_{\mu,\Lambda_n},u_j)_U.
$$
Under the constraint
$\mu(f_{\mu,\Lambda_n})=1$ we have to minimize
$$
\|f_{\mu,\Lambda_n}\|_U^2=\displaystyle{\sum_{j>n}a_j^2   }, 
$$
and by standard optimization arguments this results in 
$$
a_j= \dfrac{\mu(u_j)}{\sum_{k>n}\mu(u_k)^2},\;
\|u_{\mu,\Lambda_n}\|_U^2=\displaystyle{\left(\sum_{j>n}\mu(u_j)^2\right)^{-1}},
$$
yielding the byproduct
$$
P_{\Lambda_n}(\mu)^2=\displaystyle{ \sum_{j>n}\mu(u_j)^2}.
$$
\subsection{Splines}\label{SecSp}
These are cases where Theorem \ref{TheExTOP} applies,
if they are written
in their Hilbert space context. Power Functions can be calculated via
reciprocals of norms of Lagrangians. But the theory of this paper
allows plenty of nonstandard approaches to splines as well,
using different norms. 
\subsection{Polynomial Interpolation}\label{SecPI}
Here, the choice of the space $U$ needs special treatment, but we keep
the data being values at points $x_0,\ldots,x_n$ forming a set $X_n$ to enable
exact interpolation by polynomials of degree $n$ or order $n+1$.

The classical way to deal with this is to take $U=C^{n+1}[-1,+1]$
and to concentrate on the $(n+1)$-st derivative only, i.e.
taking the seminorm $\|f\|_U:=\|f^{(n+1)}\|_\infty$.
This brings us back to the introductory example in
Section \ref{SecIntro}. 
See how (\ref{eqPLmfUpoly}) works locally,
up to the limit $1=0\cdot\infty$ in case $x\in X_n$. Choosing other norms
will lead to different results. 
\subsection{Norms via Expansions}\label{SecNvE}
Our univariate model case here is dealing with functions
in {\tt chebfun} style (T. Driscoll et.al. \cite{driscoll-et-al:2014}),
where $U$ is a space of functions on $[-1,+1]$
having expansions
$$
f_a(x)=\displaystyle{\sum_{j=0}^\infty a_j T_j(x)  } 
$$
into Chebyshev polynomials, and where the norm takes nonnegative weights
$w_j$ 
of the coefficients, e.g.
$$
\|f_a\|_U^2=\displaystyle{\sum_{j=0}^\infty a^2_j w_j  }. 
$$
But it should be clear that one can use other expansions as well,
including multivariate cases.

This is a reproducing kernel
Hilbert space setting in disguise by
$$
(f_a,f_b)_U:=\displaystyle{\sum_{j=0}^\infty a_jb_j w_j  }
$$
and
$$
(\lambda,\mu)_{U^*}:=
\displaystyle{\sum_{j=0}^\infty \frac{\lambda(T_j)\mu(T_j)}{w_j}  }
$$
and the kernel is
$$
K(x,y):=\displaystyle{\sum_{j=0}^\infty \frac{T_j(x)T_j(y)}{w_j}  }.
$$
This holds in general, but for the Chebyshev case
this is a periodic setting in disguise, because of
$$
K(\cos \varphi,\cos \psi)
=\displaystyle{\sum_{j=0}^\infty \frac{\cos(j\varphi) \cos(j\psi)}{w_j}  }.
$$
If we would treat this like in kernel-based spaces,
interpolation would be done by linear combinations of
non-polynomial functions $K(x,x_k)$, and it would be
norm-minimal and error-minimal.
Taking $w_j=0$ for $j>n$ in the Chebyshev case, this falls back to
polynomial interpolation of degree $n$ using a basis of
$$
K(x,x_k):=\displaystyle{\sum_{j=0}^n \frac{T_j(x_k)}{w_j}T_j(x)  },
\;0\leq k\leq n.
$$
The Lagrangians and the Power Functions are invariant to basis changes,
and therefore the kernel-based viewpoint shows that Theorem
\ref{TheExTOP} holds. This opens an easy access to the Power Function via the
reciprocal of the $U$-norm of the Lagrangian.

In general, by solving 
$$
\delta_{ij}=\displaystyle{\sum_{k=0}^n a_{ik}\lambda_j(T_k)=\lambda_j(u_i)   },\;0\leq
i,j\leq n 
$$
one gets the expansion coefficients $a_{ik}$ of the Lagrangians $u_i$, and then
the reciprocal of
$$
\begin{array}{rcl}
\|u_i\|_U^2
&=&
\displaystyle{\sum_{j=0}^n a^2_{ij} w_j  } 
\end{array}
$$
gives the square of the
leave-one-out Power Function on the left-out point $x_i$.
To get a add-one-in Power function, start with $n-1$ points and add
another point $x$.

In the Chebyshev situation handling only point evaluations,
the Chebyshev-Van\-der\-mon\-de matrix $T_k(x_j)$ is particularly well-behaving
if the points are Chebyshev-distributed, as extrema of $T_n$ or zeros
of $T_{n+1}$.

But note that the above approach applies to all
expansion-based spaces where the generating functions $T_k$ 
are not the Lagrangians of the data functionals $\lambda_j$. The crucial matrix
has entries $\lambda_j(T_k)$ and is a generalized Vandermonde matrix
with possibly awful behaviour.

We want to check the add-one-in Power function
$$
\begin{array}{rcl}
  P_{a_\Lambda}(\mu)
  &=&
  \displaystyle{ \sup_{\|f\|_U\leq 1}
    \left\lvert\mu(f)-\sum_{j=0}^n\lambda_j(f)\mu(u_j)\right\rvert}\\
   &=&
  \displaystyle{ \sup_{\|f\|_U\leq 1}
    \left\lvert\sum_{k=0}^\infty a_k\left(\mu(T_k)
    -\sum_{j=0}^n\lambda_j(T_k)\mu(u_j)\right)\right\rvert}\\
   &=&
  \displaystyle{ \sup_{\|f\|_U\leq 1}
    \left\lvert\sum_{k=0}^\infty a_k\mu(\epsilon_k)\right\rvert}\\ 
\end{array}
$$
with
$$
\epsilon_k=\displaystyle{T_k
    -\sum_{j=0}^n\lambda_j(T_k)u_j},\;k\geq 0
$$
being the error on $T_k$.
We see that $\epsilon_k=0$ for $k\leq n$ and can proceed to optimize under the
constraint
$$
\|f\|_U^2=\sum_{k=0}^\infty a_k^2w_k\leq 1
$$
to get
$$
\begin{array}{rcl}
P^2_\Lambda(\mu)=\displaystyle{\sum_{k=n+1}^\infty \frac{\mu(\epsilon_k)^2}{w_k}
}\geq \dfrac{\mu(\epsilon_{n+1})^2}{w_{n+1}}. 
\end{array}
$$
The add-one-in Lagrangian $u_{\mu,\Lambda}$ arises as
$$
\epsilon_{n+1}=u_{\mu,\Lambda}\mu(\epsilon_{n+1})
$$
and if rewritten as
$$
\epsilon_{n+1}=\displaystyle{T_{n+1}
    -\sum_{k=0}^nT_k \underbrace{\sum_{j=0}^n\lambda_j(T_{n+1})a_{jk}}_{c_{k,n+1}}}
$$
the norm is
$$
\|\epsilon_{n+1}\|^2=w_{n+1}+\displaystyle{\sum_{k=0}^n c^2_{k,n+1}w_k   }
\geq w_{n+1}
$$
such that
$$
\|u_{\mu,\Lambda}\|^2=\dfrac{\|\epsilon_{n+1}\|^2}{\mu(\epsilon_{n+1})^2}
\geq \dfrac{w_{n+1}}{\mu(\epsilon_{n+1})^2}
$$
satisfies (\ref{eqPLmfU}), but not necessarily with equality, except when only
$w_{n+1}$ is nonzero.

If one uses 11 interpolation points and an additional point at -0.9056,
Figure \ref{FigBumLag} shows norm-minimal bump functions and Lagrangians.
The bump functions used Chebyshev polynomials up to order 121
to get leeway for norm minimization, and the weights
on the $T_j$ were $w_j=(j+1)^2$. The left plot is for Chebyshev points, the
right for equidistant points. The norms of Lagrangians versus bump functions
were 4.43 versus 2.43 for equidistant points, and 19.47 versus 3.3 
for Chebyshev points. The product of the Power Function with the norm of
the bump functions came out as 2.62 and 1.30 instead of one.
\begin{figure}[h]
\begin{center}
\includegraphics[height=4.0cm,width=5.5cm]{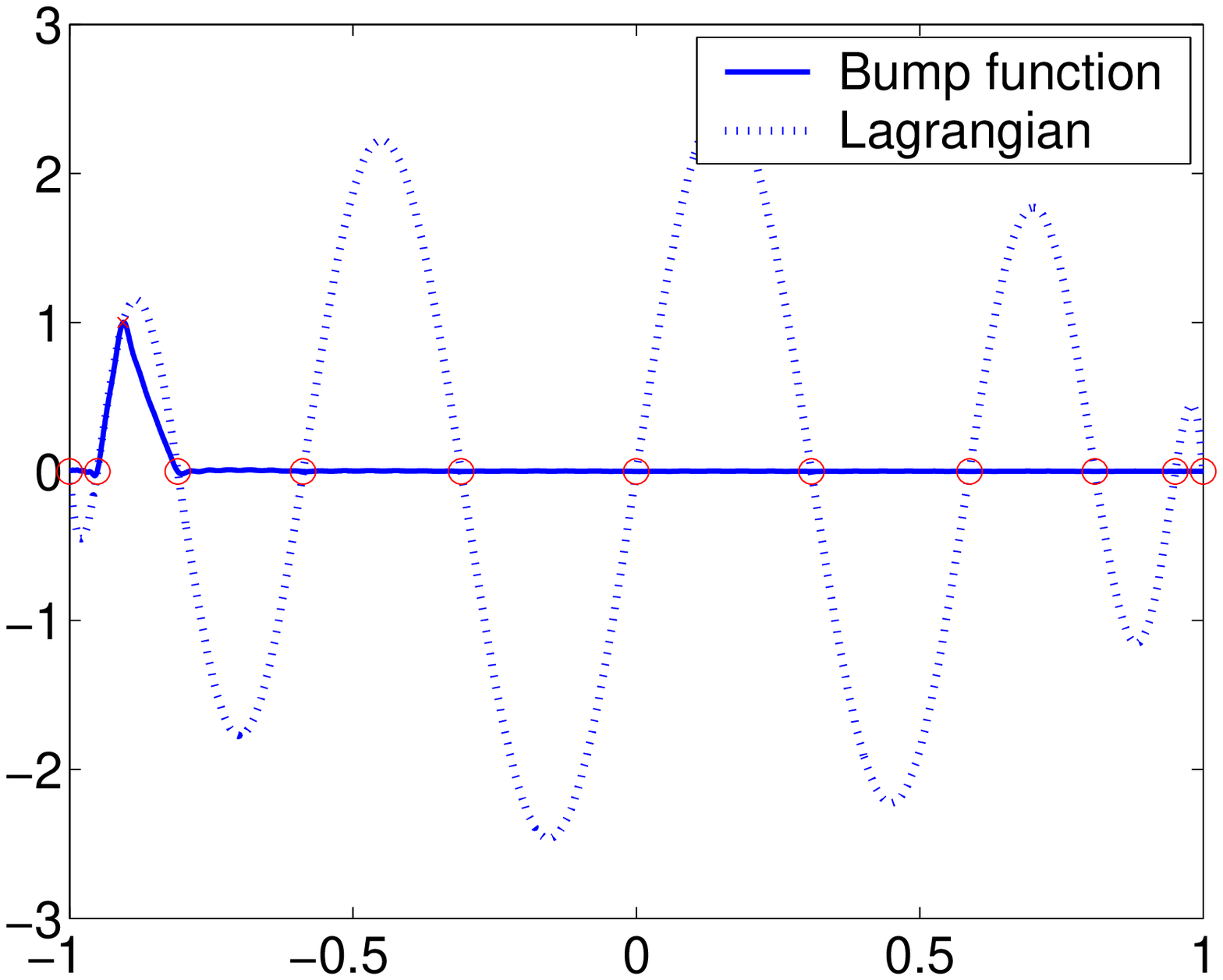} 
\includegraphics[height=4.0cm,width=5.5cm]{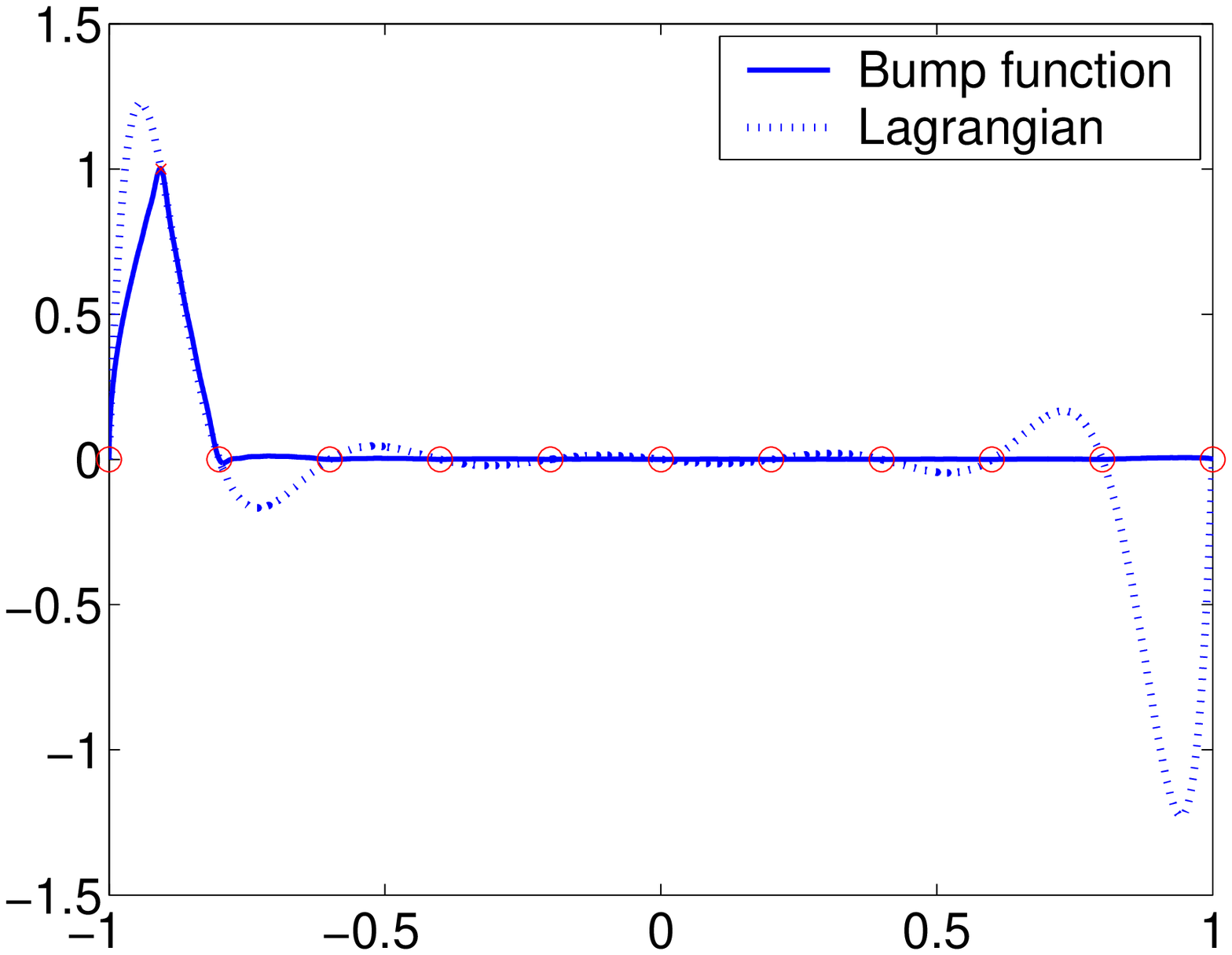} 
\end{center}
\caption{Bump functions and low-order Lagrangians for Chebyshev(left) and
equidistant points \label{FigBumLag}}
\end{figure}
\subsection{Kernel-Based Recovery Problems}\label{SecKBRP}
The goal of this section is to include the Trade-off Principle
of \cite{schaback:1995-1} as a special case,
though it looks different, considering
eigenvalues of kernel matrices there.

Assume a generalized interpolation using a set
$\Lambda:=\{\lambda_1,\ldots,\lambda_N\}$ of linearly independent
functionals and the trial space
$H_{\Lambda,K}:=\{\lambda^xK(x,\cdot)\;:\;\lambda\in \Lambda\}$
for a positive
definite kernel $K$ on a set $\Omega$.
The kernel matrix $A_{\Lambda,K}$ has entries
$$
\lambda_j^x \lambda_k^y K(x,y)=(\lambda_j,\lambda_k)_{H_K^*},\;1\leq j,k\leq N
$$
in the inner product of the
dual $H_K^*$ of the native space $H_K$ for $K$
and is positive definite. By standard arguments from Optimal Recovery in
Reproducing Kernel Hilbert Spaces, Theorem \ref{TheExTOP} holds,
and we have a Trade-off Principle in the form
(\ref{eqPLlulL}) for Lagrangians with equality.

But the result of \cite{schaback:1995-1} looks different.
To see the connection,
recall that the squared Power Function for the add-one-in
situation is the quadratic form
$$
\left(
\begin{array}{c}
  1\\
  -\mu(u_1)\\
  \vdots \\
 -\mu(u_N) 
\end{array}\right)^T
\left(\begin{array}{cccc}
(\mu,\mu) & (\mu,\lambda_1) & \ldots & (\mu,\lambda_N)\\ 
  (\lambda_1,\mu) & (\lambda_1,\lambda_1) & \ldots & (\lambda_1,\lambda_N)\\
  \vdots          &     \vdots           & \ddots &  \vdots\\  
(\lambda_N,\mu) & (\lambda_N,\lambda_1) & \ldots & (\lambda_N,\lambda_N)\\ 
\end{array}\right)
\left(
\begin{array}{c}
  1\\
  -\mu(u_1)\\
  \vdots \\
 -\mu(u_N) 
\end{array}\right)
$$
where the $u_j$ are the Lagrangians.
The proof of the trade-off principle in \cite{schaback:1995-1}
proceeds via the smallest
eigenvalue of the matrix and ignores norms of Lagrangians. Furthermore, it is just an
inequality in its original form, while Theorem \ref{TheExTOP}
yields an equation and is much more general. 

To see how \cite{schaback:1995-1}
could have proven
equality in (\ref{eqPLlulL}) more than 25 years earlier,
we look at the connection now.
This requires to identify the quadratic form with the reciprocal of 
$\|u_{\mu,\Lambda}\|_U^2$. Consider the function $u$ based on the extended set of functionals and with
coefficients $(1,-\mu(u_1),\ldots,-\mu(u_N))^T$. Then the matrix-vector product
above gives the data, and the quadratic form above
is the $U$-norm squared. The function is
$$
u(x)=\mu^y K(y,x)-\sum_j \mu(u_j)\lambda^y_jK(y,x),
$$
and it is the function where the alternative form
$$
P_{a_\Lambda}(\mu)=\sup_{\|u\|_U\leq 1,\Lambda(u)=0} \mu(f)
$$
of the Power function attains its supremum, up to a factor
\cite{DeMarchi-et-al:2005-1}.
Thus $\Lambda(u)=0$ and
$$
P_{a_\Lambda}(\mu)=\dfrac{\mu(u)}{\|u\|_U}.
$$
But the above discussion shows that $P^2_{a_\Lambda}(\mu)=\|u\|_U^2$,
proving $\mu(u)=P^2_{a_\Lambda}(\mu)$. This implies 
$u=P^2_{a_\Lambda}(\mu)u_{\mu,\Lambda}$ and
$$
P_{a_\Lambda}(\mu)^2=\|u\|_U^2=P^4_{a_\Lambda}(\mu)\|u_{\mu,\Lambda}\|_U^2
$$
to arrive finally at the Trade-off Principle in the form 
$$
1=P^2_{a_\Lambda}(\mu)\|u_{\mu,\Lambda}\|_U^2.
$$
{\bf Remark}: As long as the recovery map $a_\Lambda$,
the evaluation functionals and the chosen
space $U$ are fixed, there is no escape from the Trade-off Principle
in the above form by changes of bases, because both ingredients are basis-independent. 
This is in sharp contrast to the widespread opinion that basis changes help.
The observed large conditions of kernel matrices are a consequence
of the small Power Functions for the chosen spaces. But by changing the recovery
strategy, one can sacrifice small errors for better evaluation stability.
Sections \ref{SecUnCo}  and \ref{SecOut} will provide examples.

{\bf Remark}: Once the functionals are fixed, one can vary the kernel,
with respect to smoothness and scale. The Trade-off Principle will hold as an
equality in all cases.  
\section{Unsymmetric Case}\label{SecUnSy}
The previous two sections still used interpolation and Lagrangians.
But there are much more general cases, e.g. for PDE solving
by unsymmetric meshless methods.
In the latter case, users have no freedom to choose the data functionals,
because they are prescribed by the PDE to be solved. The functionals
will generate boundary values or values of the differential operator in the
interior. We further assume that the user prefers a certain
sort of trial functions that should finally approximate the true PDE solution
very well. In cases with
well-posedness in the sense of Real Analysis, it suffices to
come up with such a solution even if there is no uniqueness of the recovery
procedure \cite{schaback:2016-4}.

Before we look at the trial space,
recall that optimal Power Functions are purely dual objects,
$$
P_\Lambda(\mu)=\min_{a\in \mathbb{R}^{\lvert \Lambda\rvert}}\|\mu-\sum_{\lambda\in\Lambda}
a_\lambda \lambda\|_{U^*},
$$
not depending on trial spaces,
and will always outperform other solutions, error-wise.
Norm-minimal bump functions will also not be dependent on trial spaces. 
and if restricted to some trial space, their norm will not be minimal. 
In view of a Trade-off Principle,
this means that non-optimal recovery methods will sacrifice smaller errors
for larger stability.

Anyway, we now consider a set $\Lambda$ of data functionals
$\lambda_1,\ldots,\lambda_M$  and set of trial functions
$v_1,\ldots,v_N$ from some normed space $U$ of functions,
spanning a subspace $V$.
These two ingredients determine a generalized Vandermonde matrix
$A_{\Lambda,V}$ of size $M\times N$ with entries
$\lambda_j(v_k)$ that is in the theoretical
background, though certain algorithms will never
generate it as a whole. We also assume that there may be an unknown
numerical rank that limits the practical use of the matrix as is.
This occurs in plenty of kernel-based methods,
and even in square cases $M=N$
there may be a rank loss that occurs while the matrix condition
in the sense of MATLAB's {\tt condest} is still tolerable.

There are many ways to deal with this situation, and here we assume that
the practically applied
technique uses an $N\times M$ matrix $C$ that calculates coefficients
for the trial space basis for a given data vector $\Lambda(f)$.
By an $N$-vector $v$ of the $N$ basis functions, the result
is a function $v^TC\Lambda(f)$, and evaluation of a functional $\mu$ has the
error
\begin{equation}\label{eqmfmvCL}
\mu(f)-\mu(v)^TC\Lambda(f)
=\mu(f)-\sum_{j=1}^N\mu(v_j)\sum_{k=1}^M C_{jk}\lambda_k(f) 
=\mu(f)-\sum_{k=1}^M\mu(a_{k})\lambda_k(f) 
\end{equation}
for pseudo-Lagrangians
\begin{equation}\label{eqak}
a_{k}=\displaystyle{ \sum_{j=1}^Nv_jC_{jk}  },\,1\leq k\leq M 
\end{equation}
leading to the Power Function being the dual norm
\begin{equation}\label{eqPLCm}
P_{\Lambda,C}(\mu)=
\left\|\mu-\sum_{j=1}^N\mu(v_j)\sum_{k=1}^M C_{jk}\lambda_k\right\|_{U^*}.
\end{equation}
Bump functions are not necessarily connected to the trial space chosen.
If there exists a bump function $f_{\mu,\Lambda}$, the Trade-off Principle
(\ref{eqPLmfU}) applies for the above Power Function. The next section
will treat a special case in more detail, because it has a huge
background literature in applications.
\subsection{Unsymmetric Collocation}\label{SecUnCo}
An important example for solving PDEs via
a recovery of functions is unsymmetric collocation,
named after Edward Kansa \cite{kansa:1986-1}.
Here, we confine ourselves to a standard Poisson problem
(\ref{eqPoisson}) discretized as (\ref{eqPoissonDisc})
for simplicity. One chooses a reproducing kernel
Hilbert space ${\cal{H}}$ 
of functions on $\Omega$ that matches the expectable
smoothness of the solution, and implements the PDE via
test functionals, as sketched in Section \ref{SecDaF}.
The functionals in (\ref{eqPoissonDisc})
may be renamed as $\lambda_m, 1\leq m\leq M:=M_\Delta+M_\beta$
to match the notations used above. But note that $M$ will usually exceed $N$.

{\em Symmetric collocation} takes a space of trial functions where
these test functionals act on the kernel $K$,
and this is an optimal recovery
strategy \cite{schaback:2015-3} in the space ${\cal{H}}$,
with good convergence properties \cite{franke-schaback:1998-1}
\cite{franke-schaback:1998-2a}. The trade-off principle for
this was treated in section \ref{SecKBRP}.

The unsymmetric approach takes a set of trial functionals
$\tau_k=\delta_{z_k},\;1\leq k\leq N$ to generate trial functions
$$
v_k(x)=\tau_k^yK(y,x)=K(z_k,x),\;1\leq k\leq N,\;x\in \overline\Omega.
$$
The notation is now like in sections \ref{SecNvE} and  \ref{SecUnSy},
but we have not yet specified how we choose the
matrix $C$ of (\ref{eqmfmvCL}).

For calculation of the Power Function, we use
(\ref{eqmfmvCL}), define the pseudo-Lagrangians
$a_k$ from (\ref{eqak})
and get
$$
\begin{array}{rcl}
  P_{\Lambda,C}^2(\mu)
  &=&(\mu-\sum_k \mu(a_k)\lambda_k,\mu-\sum_k \mu(a_k)\lambda_k)_{U^*}\\
  &=&K_{\mu,\mu}-2b^TK_{\Lambda,\mu}+b^TK_{\Lambda,\Lambda}b 
\end{array} 
$$
in self-evident kernel matrix notation and $b=\mu(a)=\mu(C^Tv)$.
The Power Function for
symmetric collocation replaces $b$ by the solution $b^*$ of the system
$K_{\Lambda,\Lambda}b=K_{\Lambda,\mu}$ and therefore realizes the minimum
of the quadratic form over all possible vectors $b$. 

Figure \ref{FigPF2kansa01} shows squares of Power Functions
for unsymmetric collocation of a Poisson problem with
Dirichlet data on the unit
square. The setting has
121 regular interior points, 16 regular boundary points, 121
regular trial points
and uses a Matern-Sobolev kernel of order 5 at scale 1.
The matrix $C$ was
the pseudoinverse of the generalized Vandermonde matrix $A_{\Lambda,V}$.

The corresponding squares of
optimal Power Functions from symmetric collocation are in
Figure \ref{FigPF2kansa01symm}. They are not substantially smaller, just by a
factor of about $1/2$.
\begin{figure}
\begin{center}
\includegraphics[height=4.0cm,width=5.5cm]{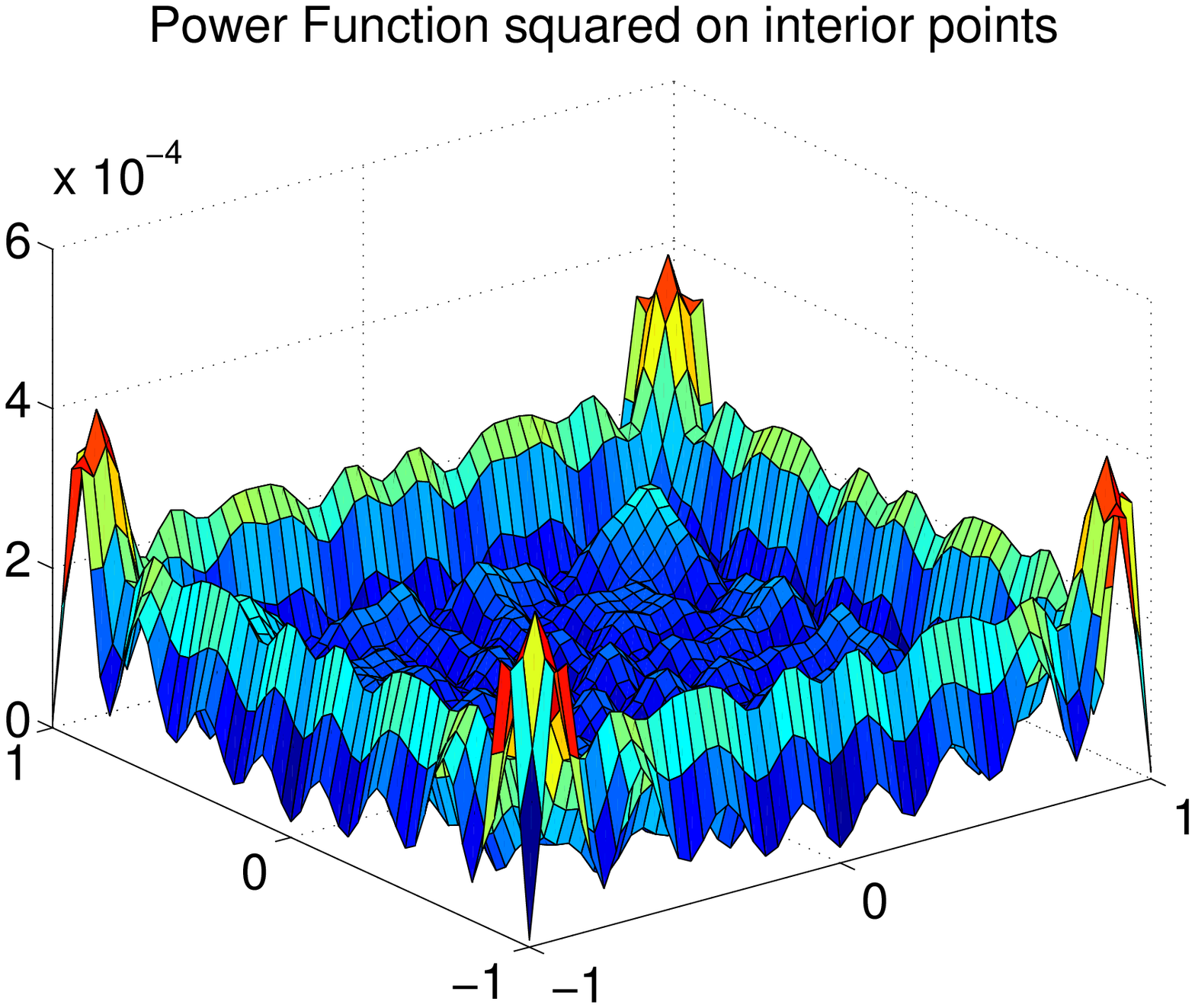} 
\includegraphics[height=4.0cm,width=5.5cm]{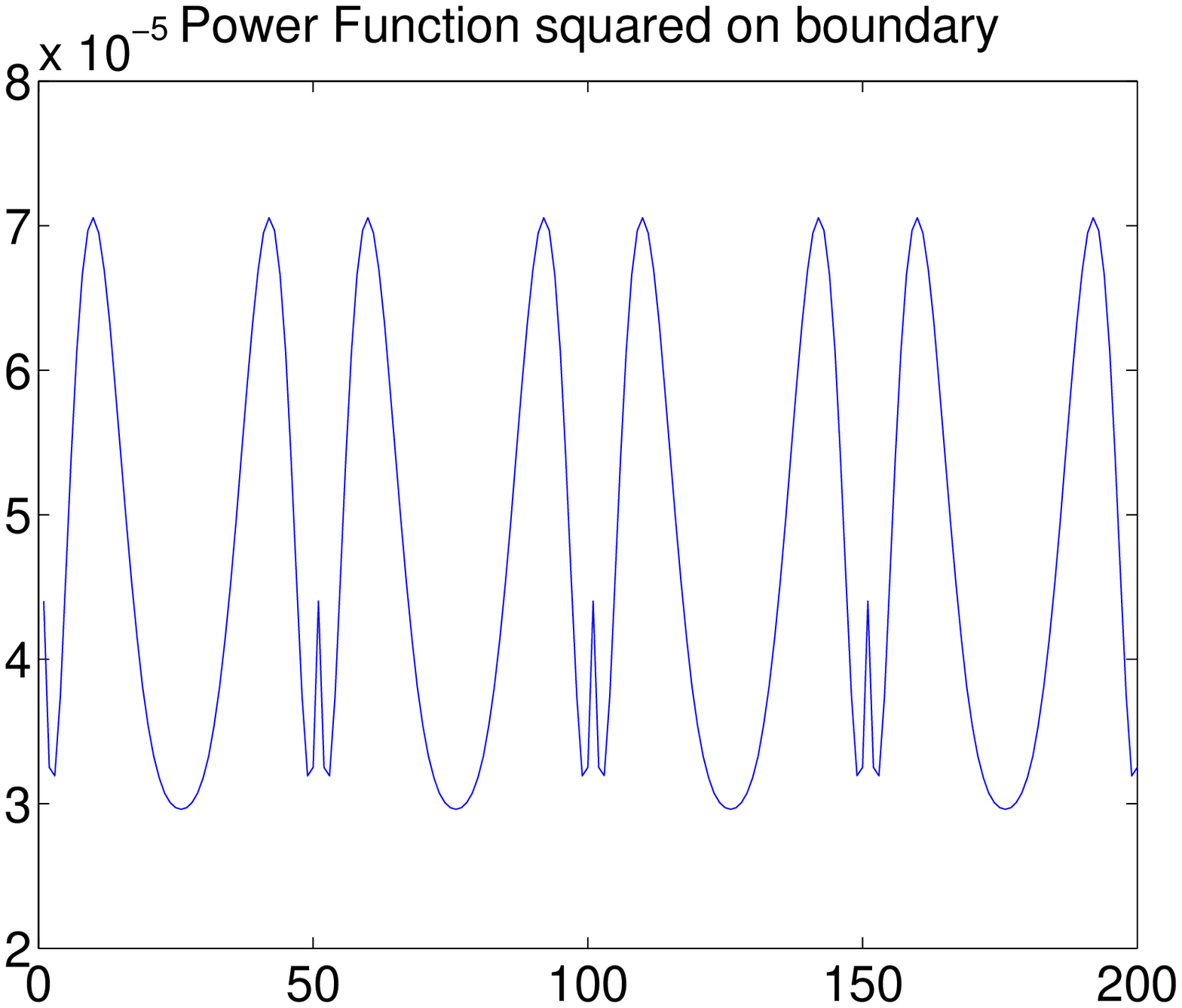}
\caption{Squares of Power Functions in interior and on
  boundary, for unsymmetric collocation \label{FigPF2kansa01}}
\end{center}  
\end{figure}
\begin{figure}
\begin{center}
\includegraphics[height=4.0cm,width=5.5cm]{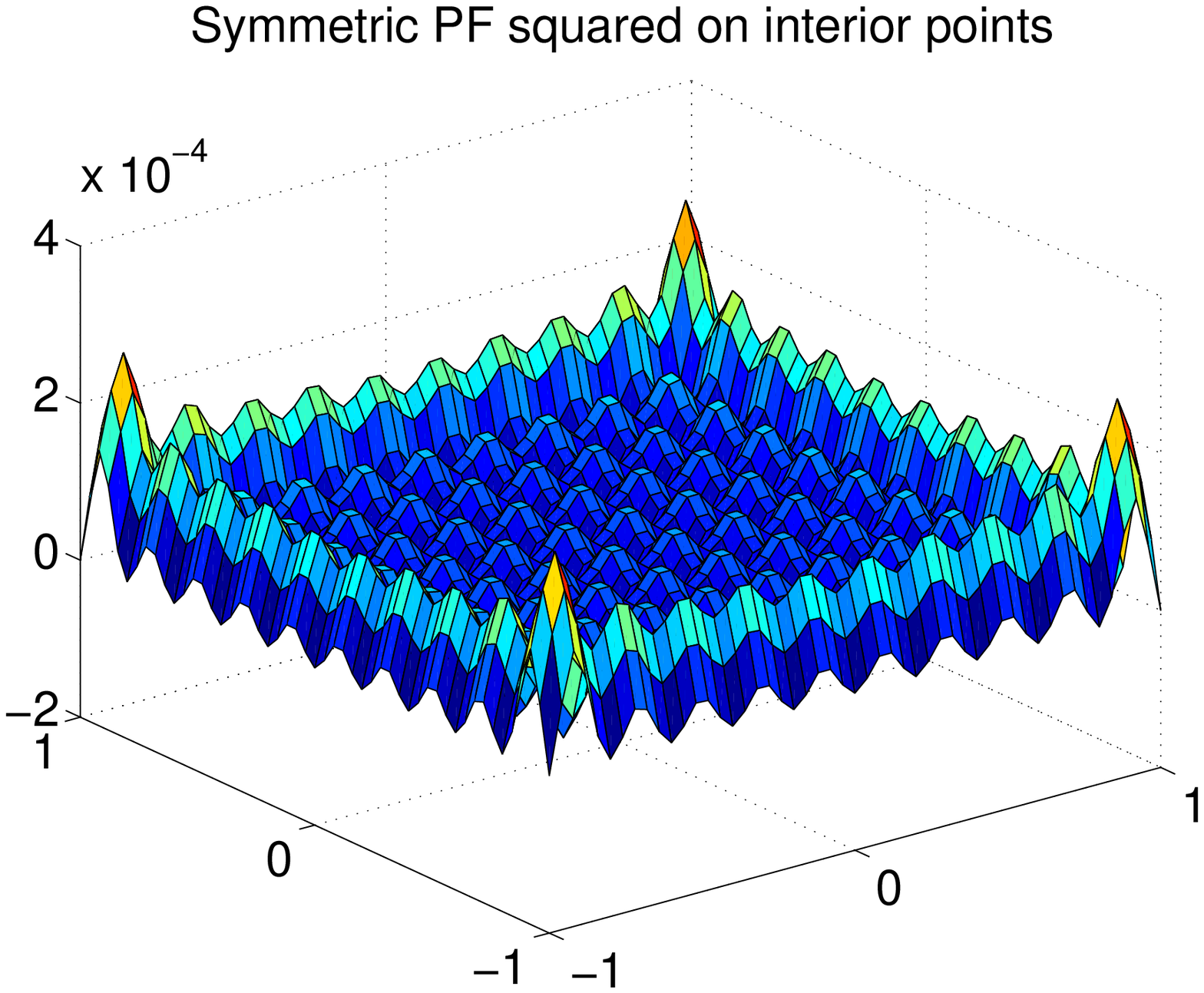} 
\includegraphics[height=4.0cm,width=5.5cm]{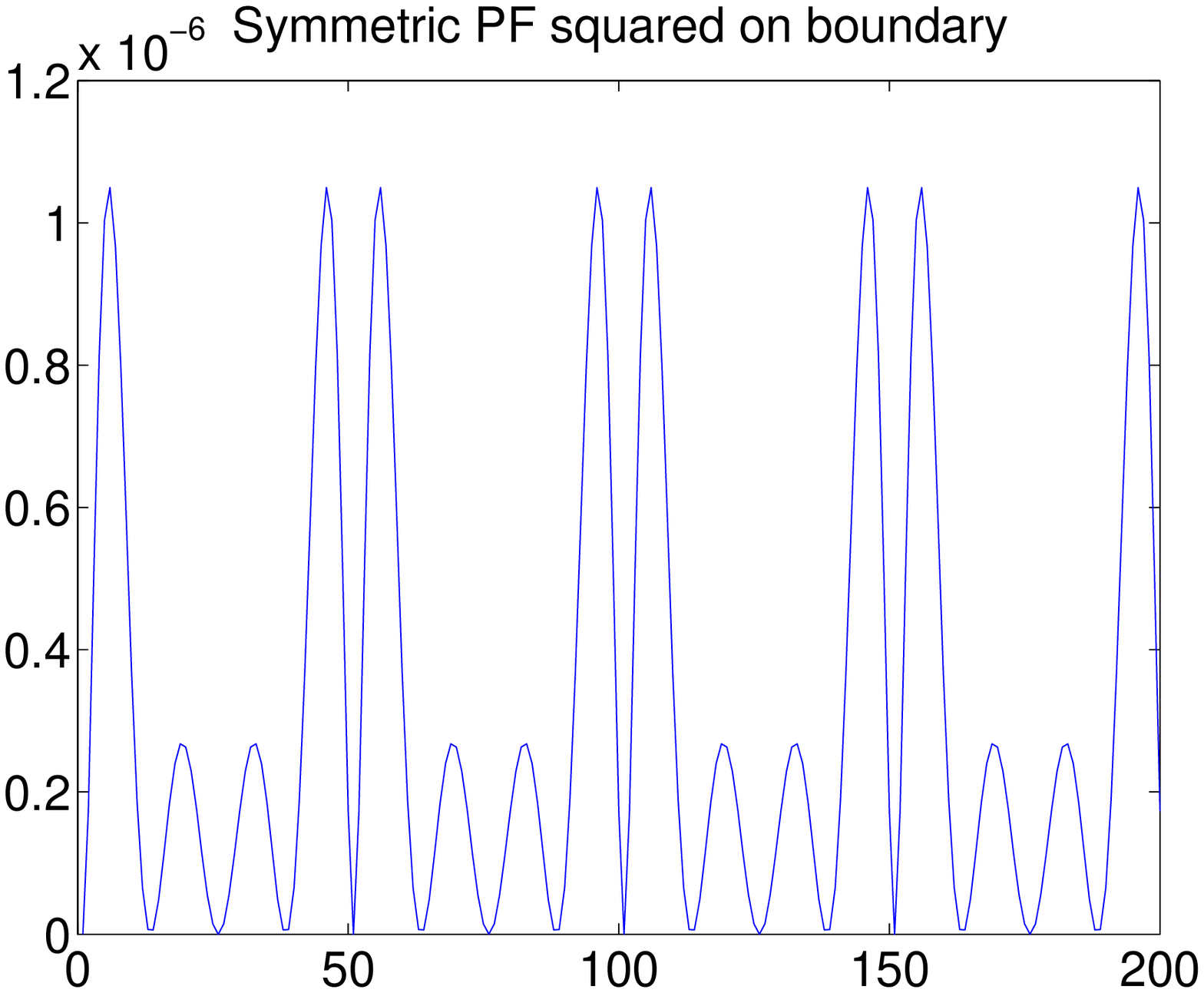}
\caption{Squares of optimal Power Functions in interior and on
  boundary, for symmetric collocation \label{FigPF2kansa01symm}}
\end{center}  
\end{figure}

Norm-minimal bump functions exist and have norms that
are related to the optimal Power Function via Theorem \ref{TheExTOP}
by equality in (\ref{eqPLlulf}) and (\ref{eqPLlulL}). They are
Lagrangians for the symmetric setting. Therefore
the right-hand sides of these inequalities get larger when
the Power Function $P_{\Lambda,C}$ is inserted. The reciprocals of
squared norms of the optimal bump functions are visualized
in Figure \ref{FigPF2kansa01symm}, because they coincide
with the square of the optimal Power Function.
Special bump functions in the
trial space of the  
unsymmetric case will usually not exist if $M> N$. 

But it may be an advantage of the unsymmetric technique that
its evaluation is based on pseudo-Lagrangians instead of Lagrangians. 

The columns of $C$ have the coefficients of the pseudo-Lagrangians
$a_k$ in the trial basis, and therefore their squared norms
$\|a_k\|_U^2$ are
in the diagonal of the matrix $C^TK_{T,T}C$ where $K_{TT}$
is the standard kernel matrix for evaluations in trial points
\cite{DeMarchi-et-al:2005-1}.
Figure
\ref{FigPsLkansa03} shows the reciprocals of these, in order to compare
with the squared optimal Power Functions in Figure \ref{FigPF2kansa01symm}.
These results come out only on the data locations, without any plotting
refinement. The values are larger by about a factor of 18 than those of the
square of the optimal Power Function,
indicating that the squared norms of the pseudo-La\-gran\-gians
are smaller than those of the Lagrangians
for the symmetric case by a factor of $1/18$. 

Summarizing, unsymmetric collocation works at a larger error level
than symmetric collocation, but
gets better evaluation stability by using pseudo-La\-gran\-gians. 

\begin{figure}
\begin{center}
\includegraphics[height=4.0cm,width=5.5cm]{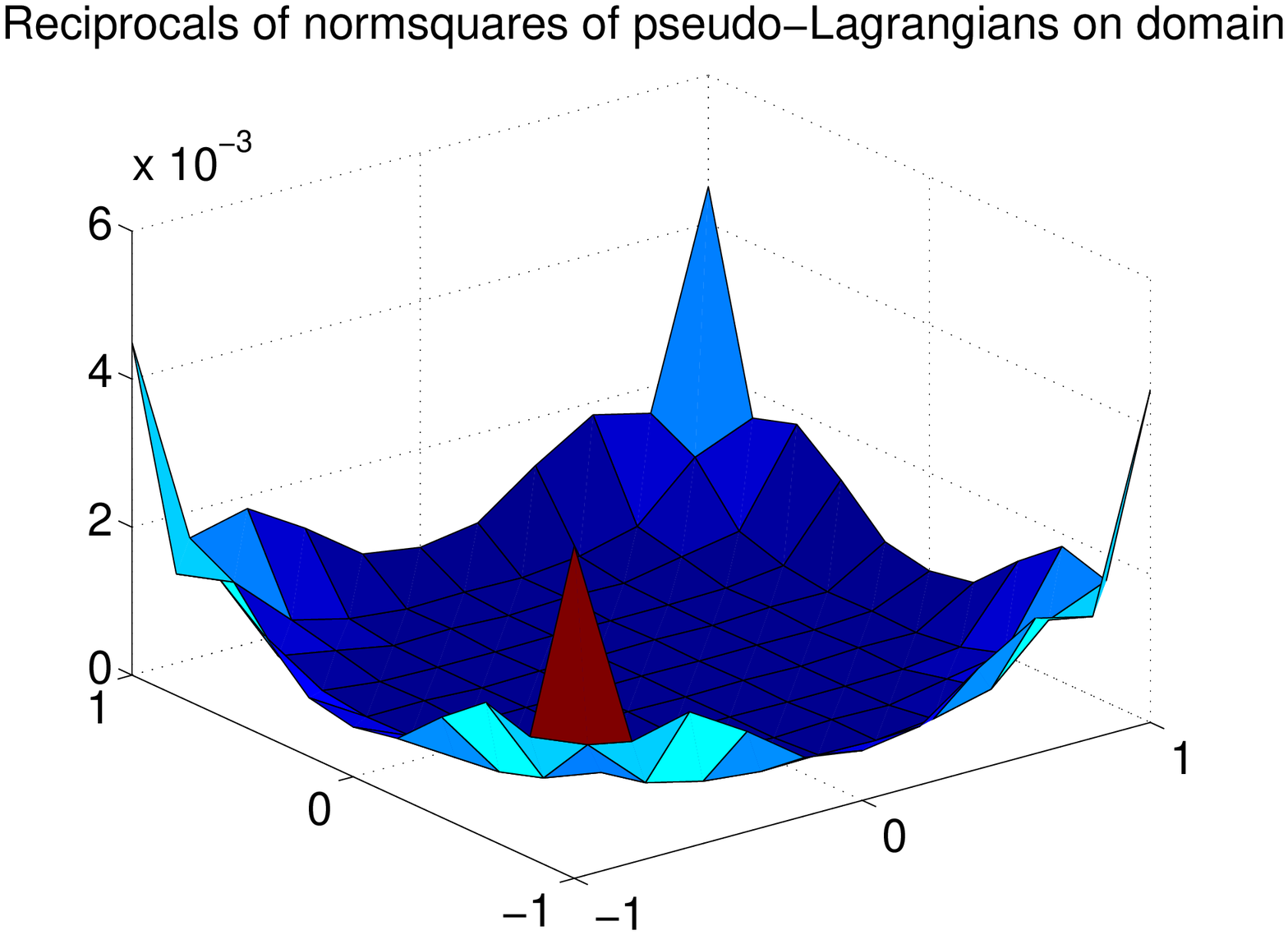} 
\includegraphics[height=4.0cm,width=5.5cm]{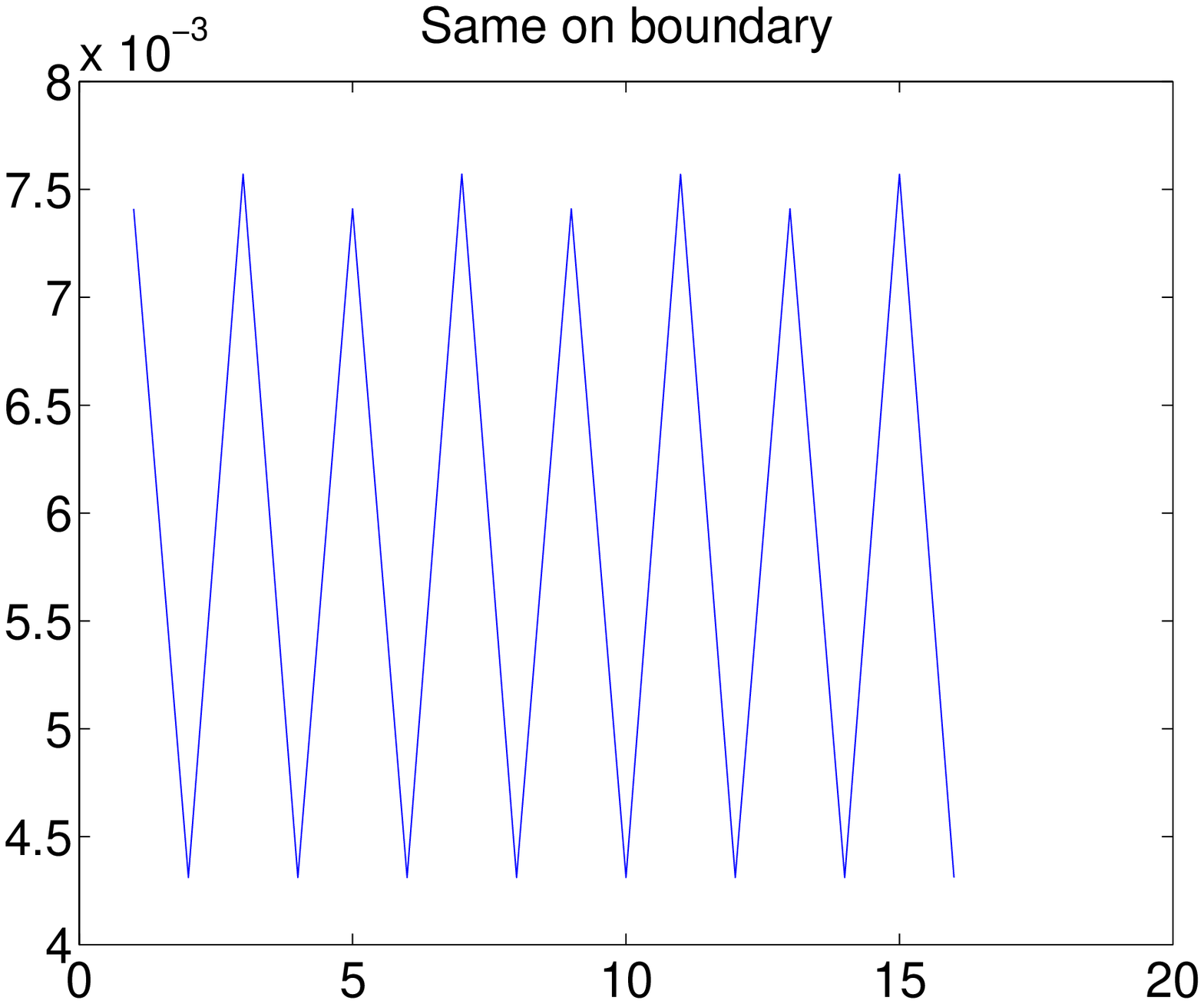}
\caption{Reciprocals of squares of norms of
  Pseudo-Lagrangians in interior and on boundary \label{FigPsLkansa03}}
\end{center}  
\end{figure}
\subsection{Non-square Linear Systems}\label{SecNLS}
Assume an overdetermined linear system $Ax\approx b$ with an
$M\times N$ matrix $A$.
The spaces $U$ and $U^*$ then are  $\mathbb{R}^M$.
After a Singular Value Decomposition,
the new $M\times N$ system is $\Sigma y\approx c$ where the diagonal of $\Sigma$
carries the $N\leq M$ nonnegative
singular values $\sigma_j,\;1\leq j\leq N$. Formally, we set
the $\sigma_j,\;N< j\leq M$ to zero. Then
$$
\begin{array}{rclll}
  \lambda_k(y) &=& \sigma_k y_k,& \hbox{ for } \sigma_k>0, & \hbox{ else } =0\\
  c_{jk} &=& \dfrac{\delta_{jk}}{\sigma_k},& \hbox{ for } \sigma_k>0, & \hbox{ else } =0\\
\end{array}
$$
if there is no regularization for small $\sigma_k$. 
By some elementary Linear Algebra,
$$
P^2_{\Lambda,C}(\mu)=\displaystyle{\sum_{\sigma_k=0}\mu_k^2   }. 
$$
All bump vectors $f_{\mu,\Lambda}$ satisfy $f_k=0$ for $\sigma_k>0$
and
$$
1=\mu(f)=\sum_{\sigma_k=0}\mu_k f_k
$$
leading to the Trade-off Principle
$$
1\leq P^2_{\Lambda,C}(\mu)\|f_{\mu,\Lambda}\|_2^2
$$
that takes the classical form (\ref{eqxy2}) here. If a Tikhonov-type
regularization uses
$$
c_{jk}(\tau) = \dfrac{\delta_{jk}}{\sigma_k+\tau}
\hbox{ for all } k,
$$
the bump functions stay the same, but the Power Function increases to
$$
P^2_{\Lambda,C(\tau)}(\mu)=\displaystyle{\sum_{k}\mu_k^2
  \dfrac{\tau^2}{(\sigma_k+\tau)^2}  }.
$$
In the square nonsingular case, the Power Function is zero
and there are no bump functions, leading back to the excluded 
$1\leq 0\cdot \infty$ situation. 
\section{Greedy Methods}\label{SecGM}
Assume the add-one-in situation, and consider an optimal 
$\mu$ to be added to $\Lambda$ for an extended problem. In view of
the Trade-off Principle, one should either take $\mu$ to maximize
$P_{a_\Lambda}(\mu)$ or to minimize $f_{\mu,\Lambda}$. In cases satisfying
Theorem \ref{TheExTOP}, these strategies coincide. This
aims at good stability and uses new functionals that cope with the current
maximal error to make it zero in the next step. For interpolation of
function values by polynomials, this leads to Leja points
(\cite{leja:1957-1}, see also the survey by St. De Marchi \cite{demarchi:2004-1}),
while for kernel-based interpolation this is the $P$-greedy method of
\cite{DeMarchi-et-al:2005-1}. Under certain additional assumptions,
these strategies are approximately optimal in the sense that they realize
$N$-widths (G. Santin and B. Haasdonk \cite{santin-haasdonk:2017-1}),
i.e. they generate trial spaces that are asymptotically optimal under all
other trial spaces of the same dimension. They can be combined with the
construction of Newton bases on-the-fly \cite{mueller-schaback:2009-1},
but we omit further details.
\section{Outlook and Open Problems}\label{SecOut}
The technique used in this paper is very elementary, and it is possible
that there are earlier results on Trade-off Principles. On instance is  
\cite{platte-et-al:2010-2} by R. Platte et.al. proving instability of
exponentially converging approximations to analytic functions. This paper proves
in general that all convergence rates have at least their exact counterpart
in rates of evaluation instability when using dual norms.

There are many more cases that fit into this paper, e.g. $h/p$ Finite Elements
or spaces of multivariate splines.
If errors are decreased by extended smoothness properties,
there always will be an increasing evaluation instability.
The connection between smoothness properties and convergence rates
is a well-known Trade-off Principle in Approximation Theory,
holding for several important cases, but a general theory
seems to be lacking.

The same holds for a hypothetical Trade-off Principle suggesting that
strongly localized methods cannot have small errors and/or large smoothness.

Handling the non-dual case is an open problem as well, in particular for
$L_\infty$ evaluation stability. 

If users have strong reasons to insist on very good accuracy
and on evaluations of high derivatives,
they have to face serious evaluation instabilities. Then it is a challenge
to cope with these, including regularizations and other changes to the
recovery map $a_\Lambda$. The literature on kernel-based
methods provides several of such techniques, e.g. Contour-Pad\'e 
\cite{fornberg-wright:2004-1}, RBF-QR \cite{fornberg-et-al:2011-1}, and 
RBF-GA \cite{fornberg-et-al:2013-1} by the group around Bengt Fornberg,
and Hilbert-Schmidt-SVD by Fasshauer/McCourt
\cite[Chapter 13]{fasshauer-mccourt:2015-1}. Greedy methods
from Section \ref{SecGM} fight evaluation 
instability by choosing functionals or evaluation points adaptively.

{\bf Acknowledgement}: The author is grateful to 
C.S. Chen of of  the University of Southern
Mississippi and Amir Noorizadegan of the National Taiwan University
for an e-mail conversation that put evaluation instability into the
focus for the Trade-off Principle. 

\backmatter


\end{document}